\documentclass[]{amsart}
\usepackage{amsfonts}
\usepackage{amssymb}
\usepackage{amsmath}
\usepackage{amsthm}
\usepackage{amscd}
\usepackage[all]{xy}
\usepackage{graphicx, color}
\usepackage{url}

\allowdisplaybreaks[0]

\newtheorem{theorem}{Theorem}[section]

\newtheorem{proposition}[theorem]{Proposition}

\theoremstyle{remark}
\newtheorem{remark}[theorem]{Remark}

\theoremstyle{definition}

\begin{document}

\title[A symmetric motion picture of the twist-spun trefoil]{A symmetric motion picture \\ of the twist-spun trefoil}
\author{Ayumu Inoue}
\address{Department of Mathematics, Tokyo Institute of Technology, Ookayama, Meguro--ku, Tokyo, 152--8551 Japan}
\email{ayumu.inoue@math.titech.ac.jp}


\begin{abstract}
With the aid of a computer, we provide a motion picture of the twist-spun trefoil which exhibits the periodicity well.
\end{abstract}

\maketitle

\section{Introduction}
\label{sec:introduction}

A \emph{$2$-knot} will be a $2$-sphere embedded into $\mathbb{R}^{4}$ locally flatly.
Although we cannot see a $2$-knot directly with our eyes, there are several ways to visualize it.
A \emph{motion picture} is one of the ways.
Roughly speaking, a motion picture is a CT scan (Computed Tomographic scan) of a $2$-knot.
It is a one-parameter family of slices of a $2$-knot, that are $0$- or $1$-dimensional objects in $\mathbb{R}^{3}$.

The \emph{$n$-twist-spun trefoil} is a periodic $2$-knot whose period is $2 \pi / n$ ($n \geq 1$).
A motion picture of the twist-spun trefoil should exhibit the periodicity.
However, we have not known such a symmetric motion picture so far.
To obtain a symmetric motion picture, we have to see a whole shape of the twist-spun trefoil.
It is difficult with paper and pencil.
We thus consider to use a computer.

In this paper, we first construct a symmetric diagram of the $2$-twist-spun trefoil using a 3D modeling software.
Here, a diagram is a projection image of a $2$-knot by a projection $\mathbb{R}^{4} \rightarrow \mathbb{R}^{3}$ with over/under information.
Slicing this diagram, we obtain a motion picture of the $2$-twist-spun trefoil which exhibits the periodicity well.
Observing this motion picture, we finally show the following theorem.

\begin{theorem}
\label{thm:main}
A motion picture depicted in Figure \ref{fig:motion_picture_of_n_twist_spun_trefoil} presents the $n$-twist-spun trefoil.
\begin{figure}[htb]
 \begin{center}
  \includegraphics[scale=0.32]{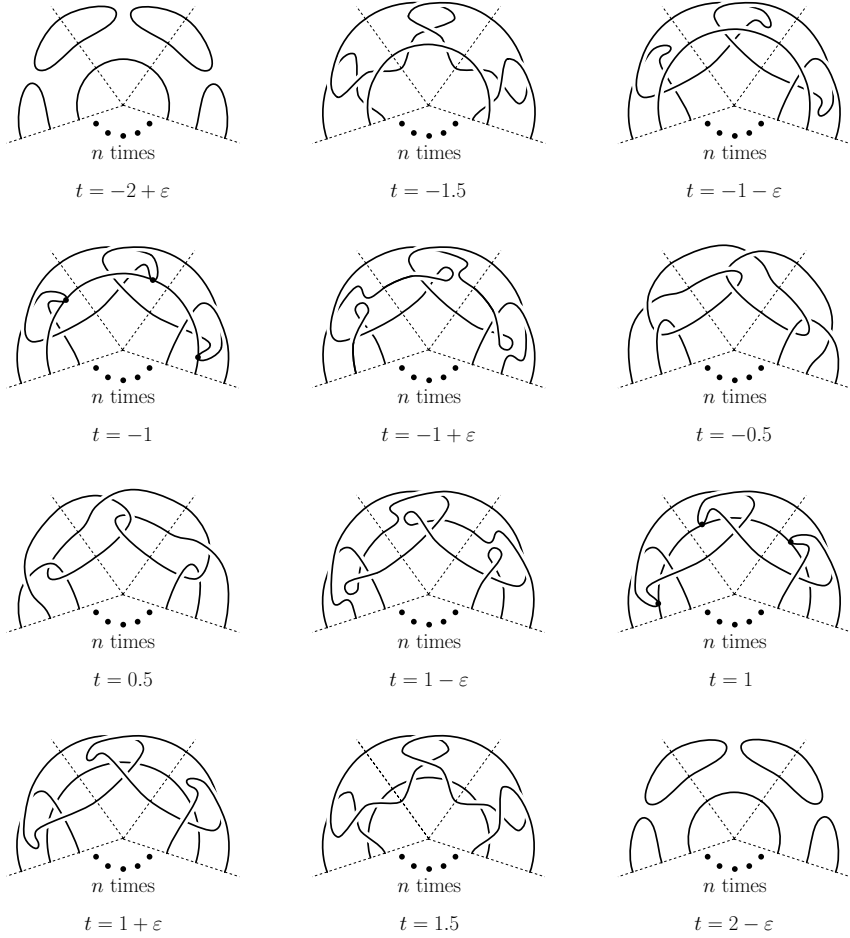}
 \end{center}
  \caption{A symmetric motion picture of the $n$-twist-spun trefoil}
 \label{fig:motion_picture_of_n_twist_spun_trefoil}
\end{figure}
\end{theorem}

Note that we can find the periodicity in this motion picture.

\section{Preliminaries}
\label{sec:preliminaries}

In this section, we recall the definitions of a diagram and a motion picture of a $2$-knot briefly.
We refer the reader \cite{CS98, Kamada02} for details of $2$-dimensional knot theory.

\subsection{Diagram}
\label{subsec:diagram}

Let $K$ be a $2$-knot and $\pi : \mathbb{R}^{4} \rightarrow \mathbb{R}^{3}$ a projection.
A point $p \in \pi(K)$ is said to be a \emph{regular point}, a \emph{double point}, a \emph{triple point}, and a \emph{branch point} if there exists a regular neighborhood $N$ of $p$ such that the triple $(N, N \cap \pi(K), p)$ is homeomorphic to the triples $(N_{1}, F_{1}, p_{1})$, $(N_{2}, F_{2}, p_{2})$, $(N_{3}, F_{3}, p_{3})$, and $(N_{4}, F_{4}, p_{4})$ depicted in Figure \ref{fig:points} respectively.
A projection $\pi$ is said to be \emph{generic} for $K$ if $\pi(K)$ consists of regular points, double points, triple points, and branch points.
It is known that there exists a generic projection for any $2$-knot.
\begin{figure}[htb]
 \begin{center}
  \includegraphics[scale=0.25]{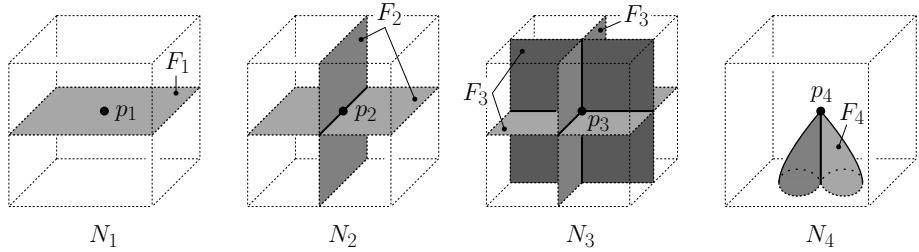}
 \end{center}
 \caption{Regular point, double point, triple point, and branch point}
 \label{fig:points}
\end{figure}

Assume that $\pi$ is generic for $K$.
A \emph{diagram} of $K$ is the image $\pi(K)$ with over/under information on double points and triple points.
We describe over/under information by removing a neighborhood of each under crossing point.
See Figures \ref{fig:parts_of_diagram} and \ref{fig:ribbon_knot}.
A subset of $\pi(K)$ consisting of all double points, triple points, and branch points is called the \emph{singularity set}.
\begin{figure}[htb]
 \begin{center}
  \includegraphics[scale=0.25]{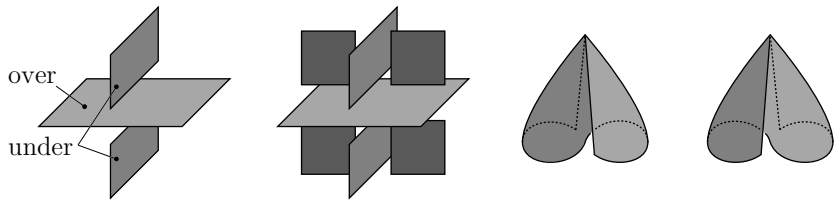}
 \end{center}
 \caption{Double point, triple point, and branch points with over/under information}
 \label{fig:parts_of_diagram}
\end{figure}
\begin{figure}[htb]
 \begin{center}
  \includegraphics[scale=0.28]{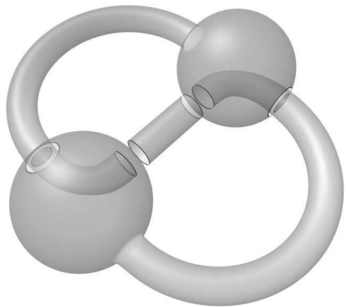}
 \end{center}
 \caption{A diagram of a $2$-knot}
 \label{fig:ribbon_knot}
\end{figure}

\subsection{Motion picture}
\label{subsec:motion_picture}

Let $f : \mathbb{R}^{4} \rightarrow \mathbb{R}$ be a map sending $(x, y, z, t)$ to $t$ and $p : \mathbb{R}^{4} \rightarrow \mathbb{R}^{3}$ a projection sending $(x, y, z, t)$ to $(x, y, z)$.
A \emph{motion picture} of a $2$-knot $K$ is a one-parameter family $\{ p(f^{-1}(t) \cap K) \}_{t \in \mathbb{R}}$.
A value $t_{0}$ of parameter $t$ is said to be \emph{regular} if the image $p(f^{-1}(t_{0}) \cap K)$ is a link or the empty set.
Here, a link is circles embedded into $\mathbb{R}^{3}$ locally flatly.
Otherwise, $t_{0}$ is said to be \emph{critical}.
Deforming $K$ slightly by an ambient isotopy if necessary, we can assume that critical values lie in $\mathbb{R}$ discretely.
Further, we can assume that one of the three situations depicted in Figure \ref{fig:critical_points} occurs locally in a motion picture if a value $t_{0}$ is critical.
Points appearing in the sequences $(1)$, $(2)$, and $(3)$ at $t = t_{0}$ are called a \emph{maximal point}, a \emph{minimal point}, and a \emph{saddle point} respectively.
Two links corresponding to regular values $t_{0}$ and $t_{1}$ are related by an ambient isotopy if there are no critical values between $t_{0}$ and $t_{1}$.
\begin{figure}[htb]
 \begin{center}
  \includegraphics[scale=0.35]{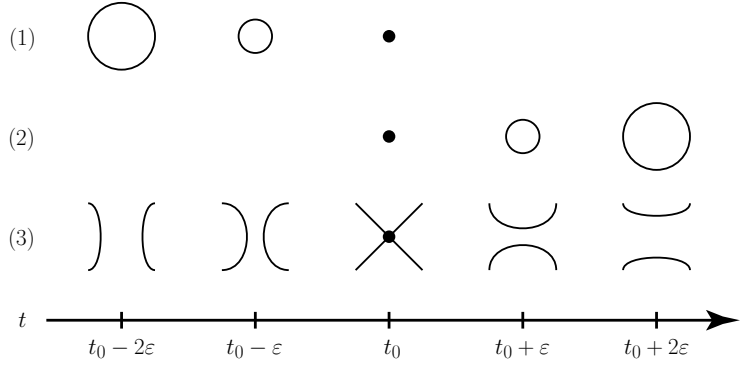}
 \end{center}
 \caption{Maximal point, minimal point, and saddle point}
 \label{fig:critical_points}
\end{figure}

A $2$-knot $K$ is said to be in a \emph{normal form} if its motion picture satisfies the following four conditions.
\begin{itemize}
 \item[1.] All minimal points appear at $t = -2$.
 \item[2.] All maximal points appear at $t = 2$.
 \item[3.] Each saddle point appears at $t = -1$ or $1$.
 \item[4.] The link corresponding to $t = 0$ is connected.
\end{itemize}
Figure \ref{fig:motion_picture} is an example of a motion picture of a $2$-knot in a normal form.
It is known that any $2$-knot can be deformed in a normal form by an ambient isotopy.
All motion pictures depicted in this paper excepting in Figure \ref{fig:motion_picture_of_two_twist_spun_trefoil} are motion pictures of $2$-knots in normal forms.
\begin{figure}[htb]
 \begin{center}
  \includegraphics[scale=0.32]{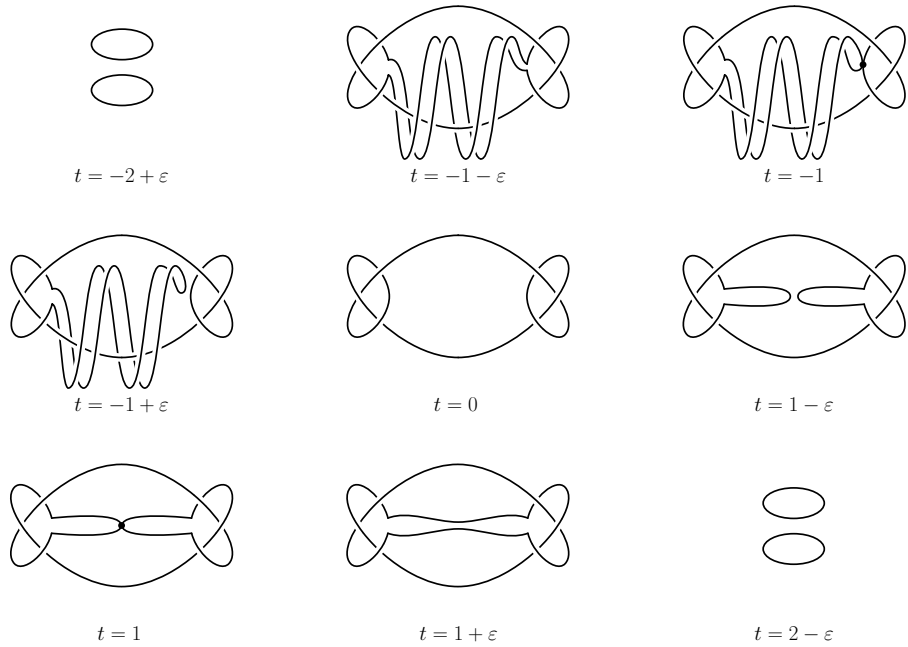}
 \end{center}
 \caption{A motion picture of a $2$-knot ($2$-twist-spun trefoil)}
 \label{fig:motion_picture}
\end{figure}

\section{Spun $2$-knot, twist-spun $2$-knot, and their known motion pictures}
\label{sec:spun_two_knot_twist_spun_two_knot_and_their_known_motion_pictures}

Artin \cite{Artin26} introduced a method to construct a $2$-knot, called a spun $2$-knot, from a knotted arc.
Zeeman \cite{Zeeman65} generalized this method, and defined a twist-spun $2$-knot.
The twist-spun trefoil is one of twist-spun $2$-knots.
In this section, we first recall the definitions of a spun $2$-knot and a twist-spun $2$-knot.
We next consider their known motion pictures.

\subsection{Spun $2$-knot and twist-spun $2$-knot}
\label{subsec:spun_two_knot_and_twist_spun_two_knot}

Let $\mathbb{R}^{3}_{+} = \{ (x, y, z) \in \mathbb{R}^{3} \mid z \geq 0 \}$ be the upper half space.
A \emph{properly embedded arc} is an arc embedded into $\mathbb{R}^{3}_{+}$ locally flatly and intersecting with $\partial \mathbb{R}^{3}_{+}$ transversally only at its end points.
See Figure \ref{fig:properly_embedded_arc}.
Let $k$ be a properly embedded arc.
Spin $\mathbb{R}^{3}_{+}$ 360 degrees in $\mathbb{R}^{4}$ along $\partial \mathbb{R}^{3}_{+}$.
Remark that $(x, y, z \cos \theta, z \sin \theta)$ is a coordinate of $\mathbb{R}^{4}$ ($x, y, z \in \mathbb{R}$, $z \geq 0$, $0 \leq \theta < 2 \pi$).
Then the locus of $k$ yields a $2$-knot.
See Figure \ref{fig:spun_knot}.
We call this $2$-knot a \emph{spun $2$-knot} of $k$.
\begin{figure}[htb]
 \begin{center}
  \includegraphics[scale=0.32]{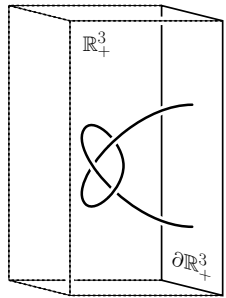}
 \end{center}
  \caption{A properly embedded arc}
 \label{fig:properly_embedded_arc}
\end{figure}
\begin{figure}[htb]
 \begin{center}
  \includegraphics[scale=0.50]{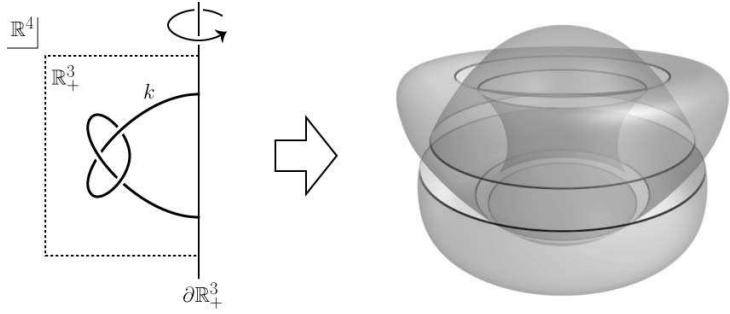}
 \end{center}
  \caption{Spin $\mathbb{R}^{3}_{+}$ 360 degrees in $\mathbb{R}^{4}$ along $\partial \mathbb{R}^{3}_{+}$. Then we obtain a spun $2$-knot which is the locus of a properly embedded arc $k$.}
 \label{fig:spun_knot}
\end{figure}

Let $B$ be a $3$-ball in $\mathbb{R}^{3}_{+}$ which contains a knotted part of $k$.
While spinning $\mathbb{R}^{3}_{+}$ along $\partial \mathbb{R}^{3}_{+}$, twist $B$ $n$-times as depicted in Figure \ref{fig:twist_spun_knot} ($n \geq 1$).
Then the locus of $k$ also yields a $2$-knot.
We call this $2$-knot an \emph{$n$-twist-spun $2$-knot} of $k$.
A spun $2$-knot and twist-spun $2$-knots of a properly embedded arc are not ambient isotopic to each other in general.
Zeemann \cite{Zeeman65} showed that any $1$-twist-spun $2$-knot is trivial.
That is, it bounds an embedded $3$-ball in $\mathbb{R}^{4}$.
\begin{figure}[htb]
 \begin{center}
  \includegraphics[scale=0.32]{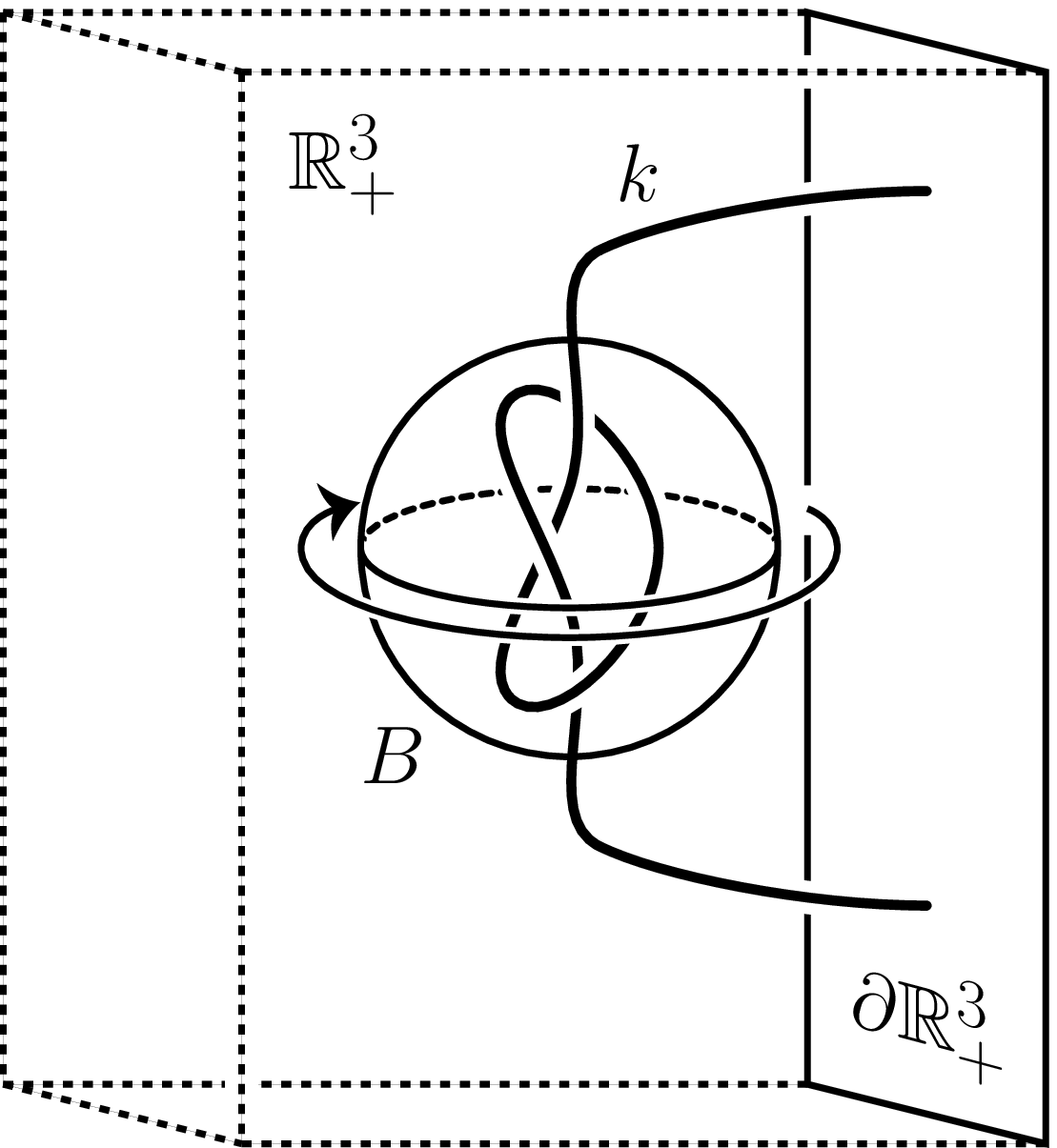}
 \end{center}
  \caption{Twist a $3$-ball $B$ containing a knotted part of a properly embedded arc $k$ $n$-times in $\mathbb{R}^{3}_{+}$.}
 \label{fig:twist_spun_knot}
\end{figure}

The properly embedded arc depicted in Figure \ref{fig:properly_embedded_arc} is called a (long) \emph{trefoil}.
The \emph{spun trefoil} and the \emph{twist-spun trefoil} are a spun $2$-knot and a twist-spun $2$-knot of a trefoil respectively.

\subsection{Known motion pictures}
\label{subsec:known_motion_pictures}

Consider parallel vertical hyperplanes depicted in Figure \ref{fig:motion_picture_of_spun_knot} that intersect with a spun $2$-knot.
We assume that the hyperplanes are parametrized by $t \in [-2, 2]$.
It is not difficult to see cross-sections of a spun $2$-knot taken along the hyperplanes.
We thus have a motion picture consisting of the cross-sections.
Figure \ref{fig:motion_picture_of_spun_trefoil} is a motion picture of the spun trefoil.
\begin{figure}[htb]
 \begin{center}
  \includegraphics[scale=0.20]{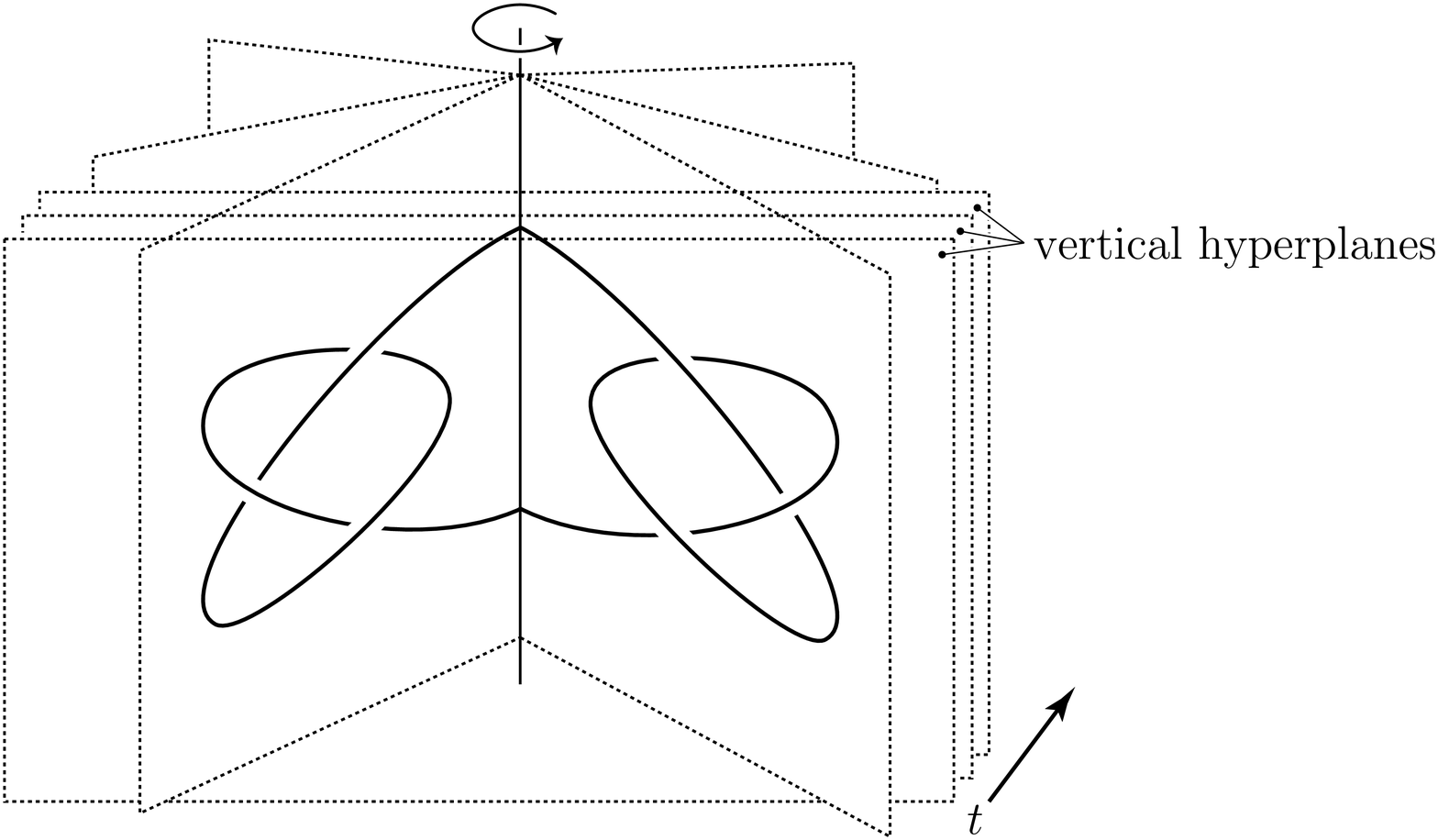}
 \end{center}
  \caption{Parallel vertical hyperplanes that intersect with a spun $2$-knot}
 \label{fig:motion_picture_of_spun_knot}
\end{figure}
\begin{figure}[htb]
 \begin{center}
  \includegraphics[scale=0.28]{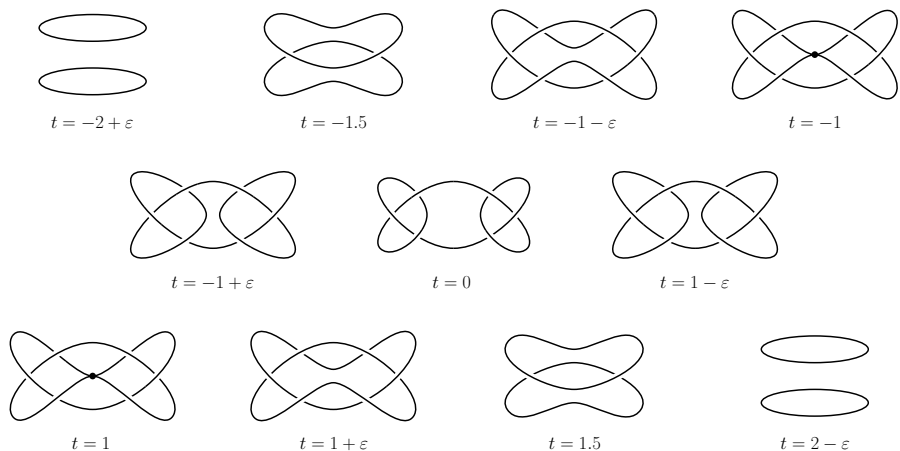}
 \end{center}
  \caption{A motion picture of the spun trefoil}
 \label{fig:motion_picture_of_spun_trefoil}
\end{figure}

Modifying this motion picture, we can obtain a motion picture of a twist-spun $2$-knot as follows.
For simplicity, we consider a motion picture of the $n$-twist-spun trefoil in the remaining of this section.
Add a motion in the motion picture of the spun trefoil at $t = -1$ that corresponds to $n$-twist of the left half.
See Figure \ref{fig:twist_in_motion}.
The gray band depicts a saddle point which is stretched by twisting.
This modified motion picture obviously presents the $n$-twist-spun trefoil.
Deforming this $2$-knot in a normal form by an ambient isotopy, we have a motion picture of the $n$-twist-spun trefoil depicted in Figure \ref{fig:motion_picture} (in the case $n = 2$).
\begin{figure}[htb]
 \begin{center}
 \includegraphics[scale=0.32]{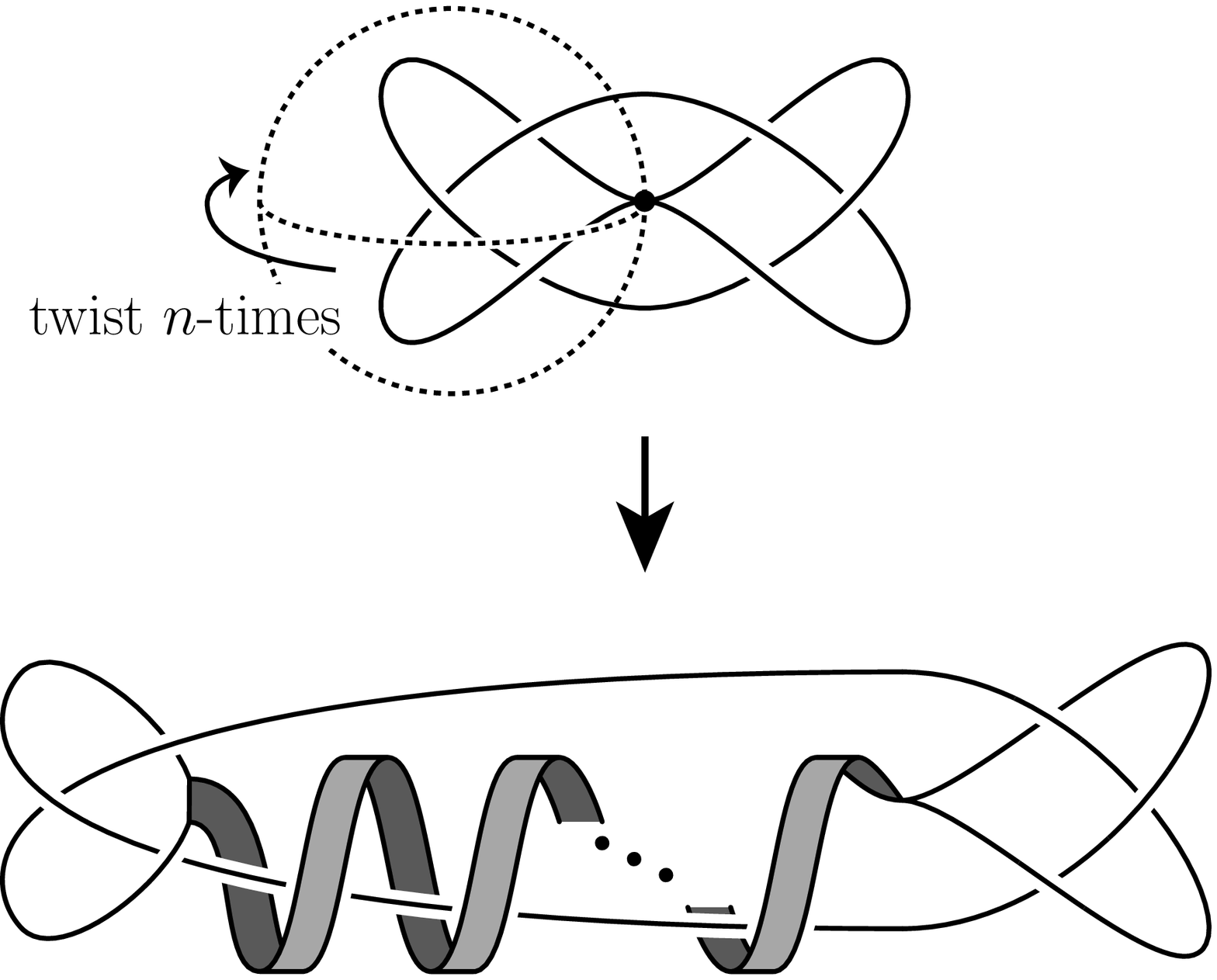}
 \end{center}
  \caption{Twist the left half $n$-times at $t = -1$.}
 \label{fig:twist_in_motion}
\end{figure}

This motion picture of the twist-spun trefoil does not exhibit the periodicity well.
Because we have cut the twist-spun trefoil along parallel vertical hyperplanes.
If we cut the twist-spun trefoil along parallel horizontal hyperplanes, then we shall obtain a motion picture which exhibits the periodicity.
However, it is difficult to see cross-sections of the twist-spun trefoil taken along parallel horizontal hyperplanes with paper and pencil.
We thus use a computer.

\section{Diagram of the $2$-twist-spun trefoil}
\label{sec:diagram_of_the_two_twist_spun_trefoil}

In this section, we construct a diagram of the $2$-twist-spun trefoil, instead of the $2$-twist-spun trefoil itself, using a 3D modeling software.
Satoh and Shima drew a picture of a regular neighborhood of the singularity set of a $2$-twist-spun trefoil diagram locally in Figure 13 of \cite{SS04}.
Using this picture as a reference, we first construct a regular neighborhood of the singularity set of a $2$-twist-spun trefoil diagram.
We then pay attention that it exhibits the periodicity well.
As a result, we obtain a part of a diagram depicted in Figure \ref{fig:frame_of_two_twist_spun_trefoil}.
See also a movie at \url{http://www.youtube.com/watch?v=BB1B6hDJga8} to watch this part of a diagram in some directions.
\begin{figure}[htb]
 \begin{center}
  \includegraphics[scale=0.55]{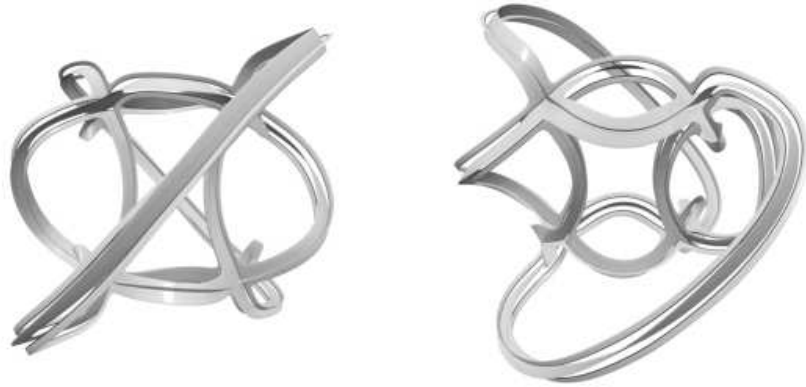}
 \end{center}
 \caption{A regular neighborhood of the singularity set of a diagram of the $2$-twist-spun trefoil. Because of the difference in settings, this is a mirror image of the one depicted in \cite{SS04}.}
 \label{fig:frame_of_two_twist_spun_trefoil}
\end{figure}

We next fill the remaining of a diagram.
In the process, we have to create no more double points, triple points, and branch points.
Further, a result should be a broken immersed $2$-sphere.
Then we obtain a diagram depicted in Figure \ref{fig:diagram_of_two_twist_spun_trefoil}.
See also a movie at \url{http://www.youtube.com/watch?v=-obouoE3DVA} to watch this diagram in some directions.
There may be many ways to execute this process.
We thus have to show the following proposition.
\begin{figure}[htb]
 \begin{center}
  \includegraphics[scale=0.60]{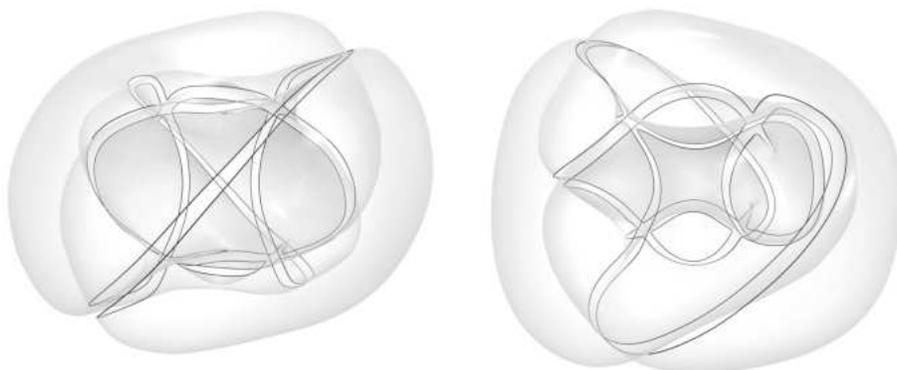}
 \end{center}
 \caption{A diagram of the $2$-twist-spun trefoil}
 \label{fig:diagram_of_two_twist_spun_trefoil}
\end{figure}

\begin{proposition}
\label{prop:is_diagram_of_two_twist_spun_trefoil}
The diagram depicted in Figure \ref{fig:diagram_of_two_twist_spun_trefoil} presents the $2$-twist-spun trefoil.
\end{proposition}

\begin{proof}
Consider vertical planes depicted in Figure \ref{fig:planes_extending_radially} that extend radially from a line passing through the center of the diagram.
We assume that the planes are parametrized by $\theta \in [0, 2 \pi)$.
Let us observe cross-sections of the diagram taken along the planes.
It is easy in a computer.
Assume that the cross-sections are also parametrized by $\theta$.
It is sufficient that we examine cross-sections between $\theta = 0$ and $\theta = \pi$, because the diagram is periodic.
That cross-sections are depicted in Figure \ref{fig:proof_is_diagram_of_two_twist_spun_trefoil_1}.
See also a movie at \url{http://www.youtube.com/watch?v=-OEMiR43iIs} for more details.
A transformation sequence of cross-sections between $\theta = 0$ and $\theta = \pi$ is depicted in Figure \ref{fig:proof_is_diagram_of_two_twist_spun_trefoil_2}.
The moves from one to the next are indicated by gray arrows etc.
We claim that this sequence presents a motion corresponding to one twist of a trefoil.
Then the diagram presents the $2$-twist-spun trefoil (see \cite{Satoh02, SS04}).
\begin{figure}[htb]
 \begin{center}
  \includegraphics[scale=0.35]{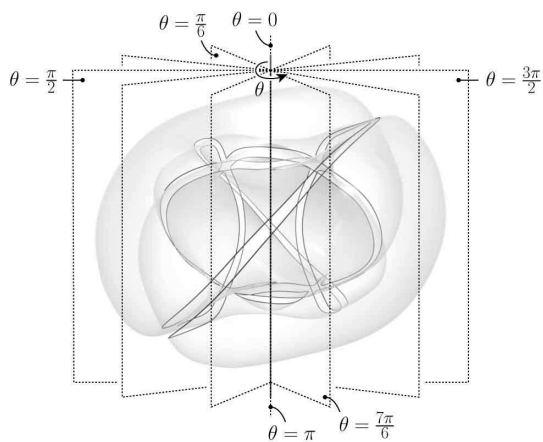}
 \end{center}
 \caption{Vertical planes extending radially from the center line of the $2$-twist-spun trefoil diagram}
 \label{fig:planes_extending_radially}
\end{figure}
\begin{figure}[htb]
 \begin{center}
  \includegraphics[scale=0.67]{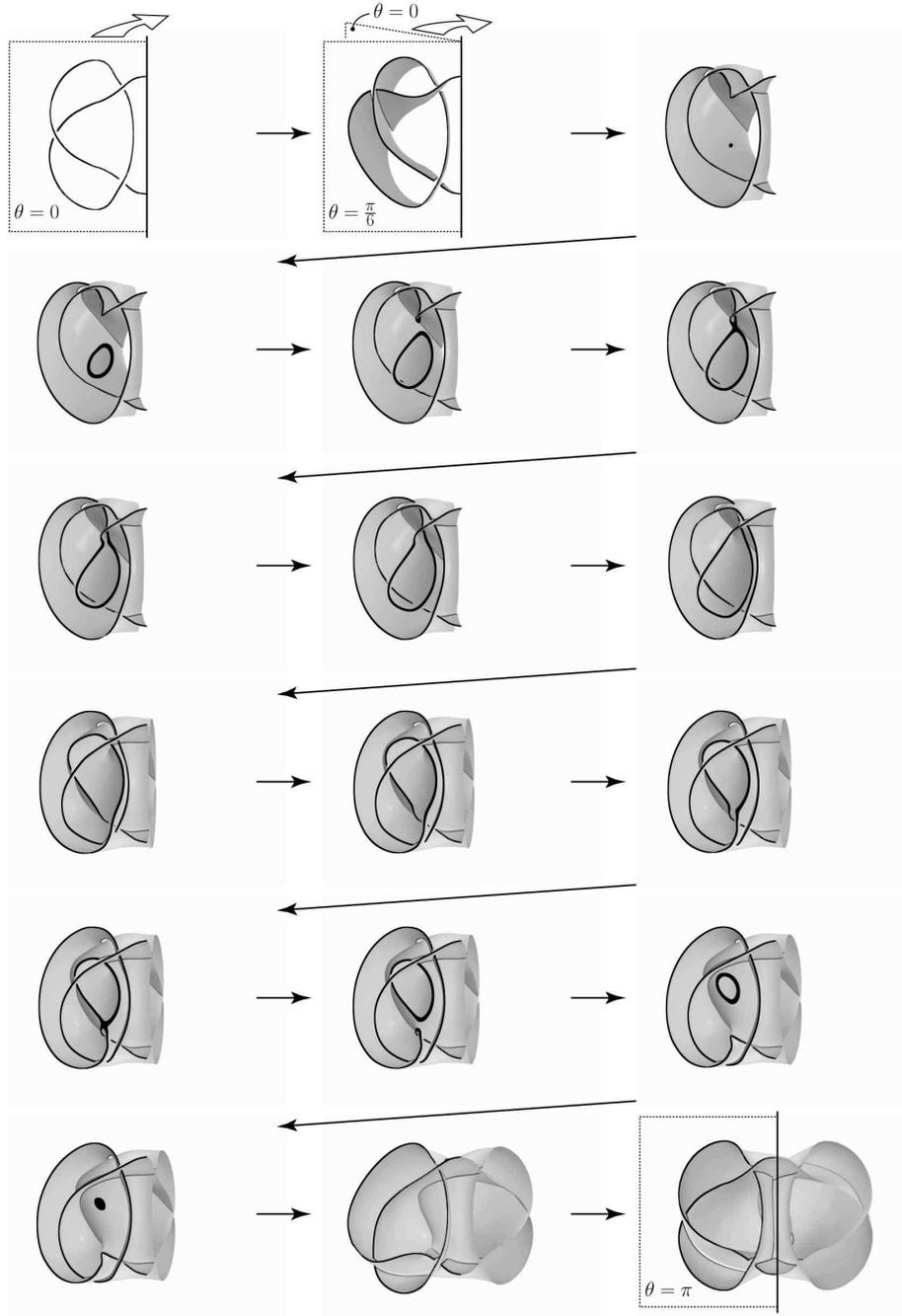}
 \end{center}
 \caption{Cross-sections of the diagram between $\theta = 0$ and $\theta = \pi$}
 \label{fig:proof_is_diagram_of_two_twist_spun_trefoil_1}
\end{figure}
\begin{figure}[htb]
 \begin{center}
  \includegraphics[scale=0.21]{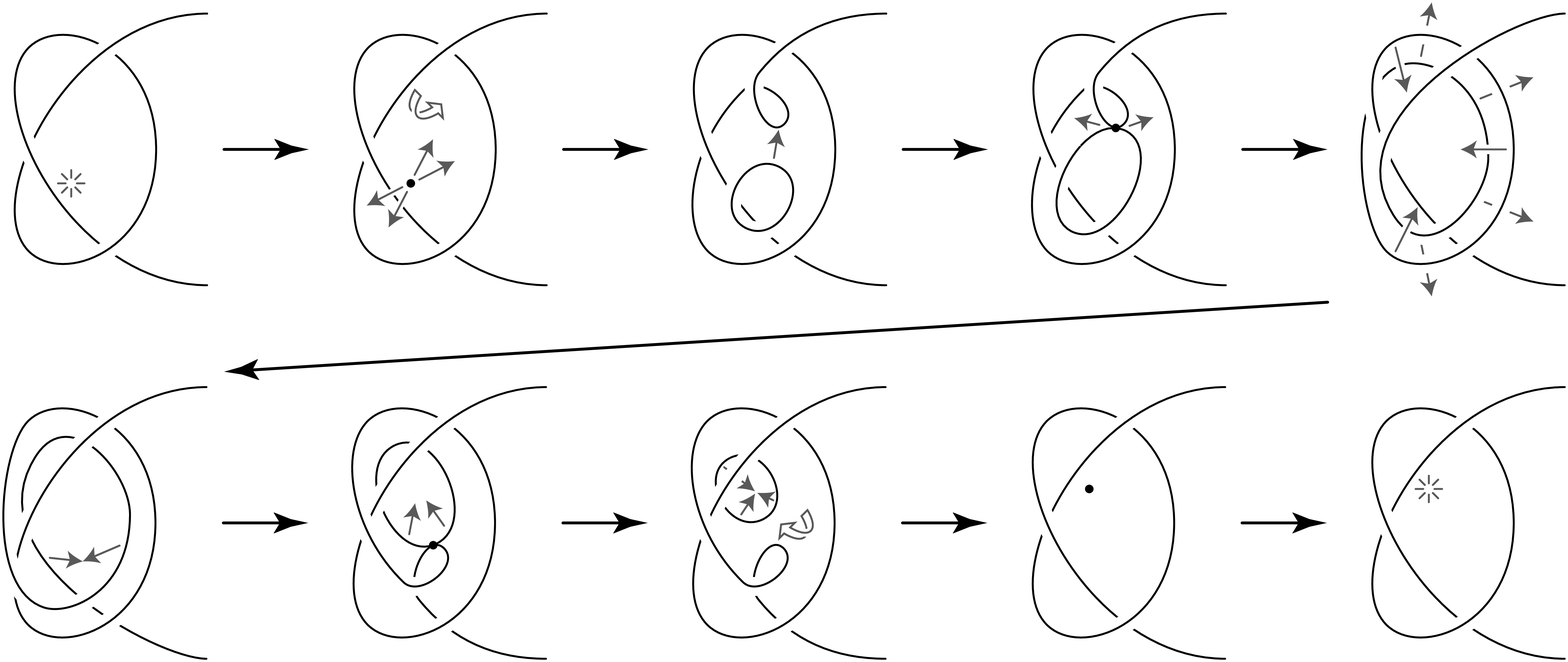}
 \end{center}
 \caption{A transformation sequence of cross-sections between $\theta = 0$ and $\theta = \pi$}
 \label{fig:proof_is_diagram_of_two_twist_spun_trefoil_2}
\end{figure}

To prove the claim, consider two sequences depicted in Figures \ref{fig:proof_is_diagram_of_two_twist_spun_trefoil_3} and \ref{fig:proof_is_diagram_of_two_twist_spun_trefoil_4}.
Satoh and Shima \cite{SS04} showed that the sequence depicted in Figure \ref{fig:proof_is_diagram_of_two_twist_spun_trefoil_3} presents a motion corresponding to one twist of a trefoil.
It is routine to check that sequences depicted in Figures \ref{fig:proof_is_diagram_of_two_twist_spun_trefoil_3} and \ref{fig:proof_is_diagram_of_two_twist_spun_trefoil_4} present a same motion as a whole.
That is, the difference does not affect to present a motion corresponding to one twist of a trefoil.
Further, sequences depicted in Figures \ref{fig:proof_is_diagram_of_two_twist_spun_trefoil_4} and \ref{fig:proof_is_diagram_of_two_twist_spun_trefoil_2} present a same motion as a whole.
Hence, the claim is true.
\begin{figure}[htb]
 \begin{center}
  \includegraphics[scale=0.21]{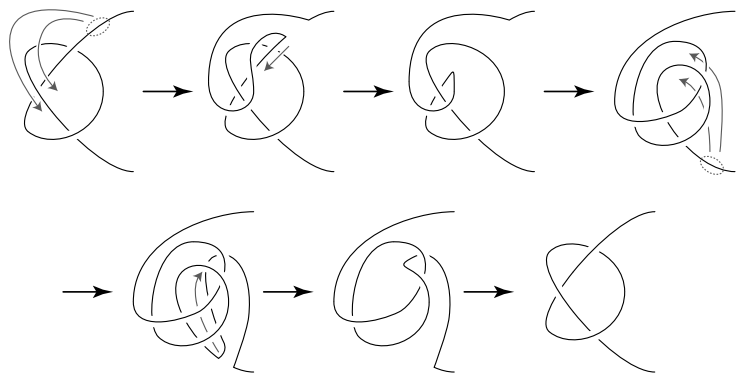}
 \end{center}
 \caption{A sequence which presents a motion corresponding to one twist of a trefoil (1)}
 \label{fig:proof_is_diagram_of_two_twist_spun_trefoil_3}
\end{figure}
\begin{figure}[htb]
 \begin{center}
  \includegraphics[scale=0.21]{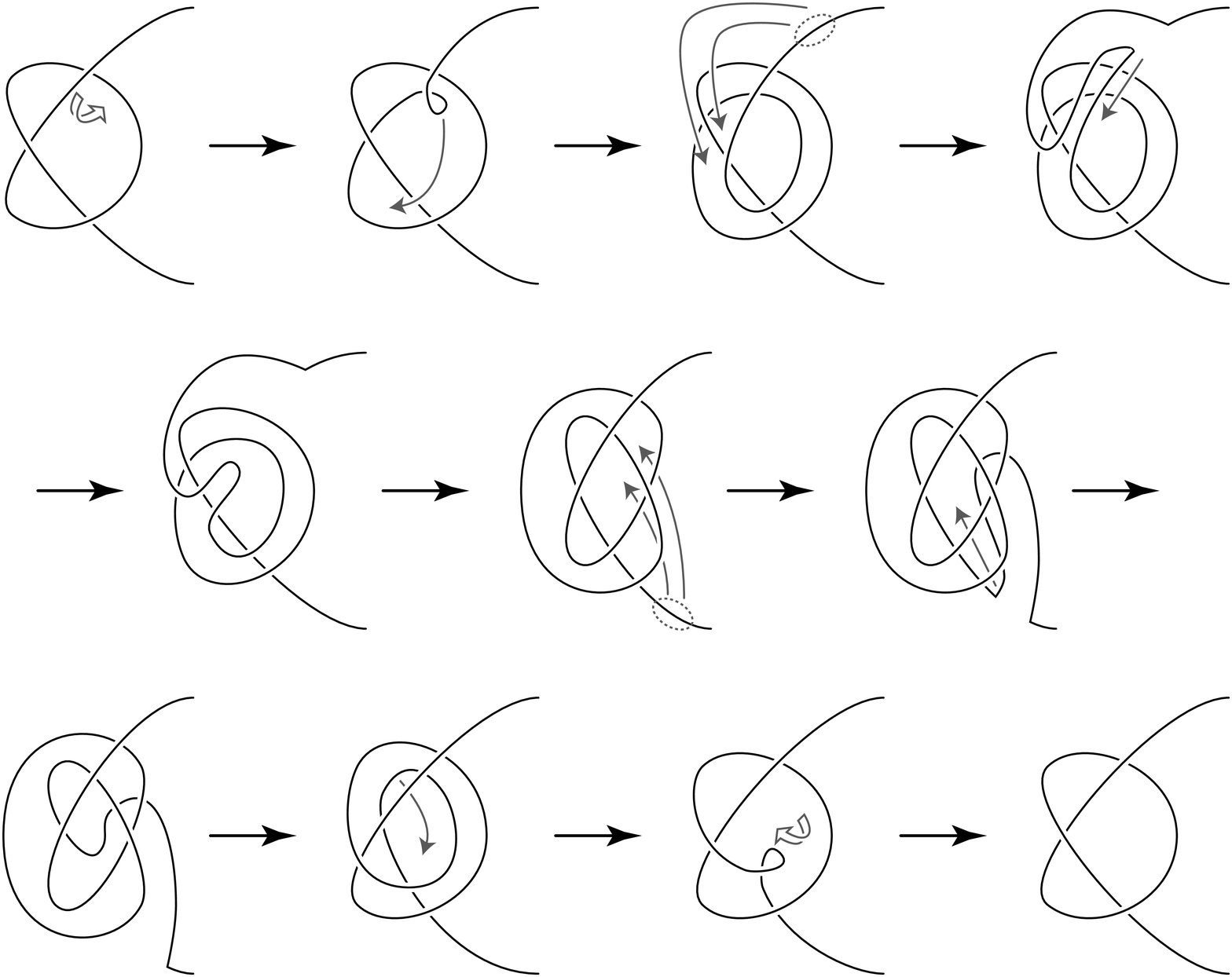}
 \end{center}
 \caption{A sequence which presents a motion corresponding to one twist of a trefoil (2)}
 \label{fig:proof_is_diagram_of_two_twist_spun_trefoil_4}
\end{figure}
\end{proof}

\begin{remark}
This diagram of the $2$-twist-spun trefoil has four triple points.
Further, it consists of four connected components\footnote{\hskip 0.1em See a movie at \url{http://www.youtube.com/watch?v=u1j8JLP4bUU} to watch a diagram which is color-coded according to connected components.}.
Therefore, this diagram realizes the triple point number and the sheet number of the $2$-twist-spun trefoil (see \cite{Satoh07, SS04}).
\end{remark}

\section{Proof of Theorem \ref{thm:main}}
\label{sec:proof_of_main_theorem}

\begin{proof}[Proof of Theorem \ref{thm:main}]
Let us cut the diagram of the $2$-twist-spun trefoil, which we have constructed in the previous section, along horizontal planes as depicted in Figure \ref{fig:cut_diagram_along_horizontal_planes}.
It is easy in a computer.
We assume that the planes are parametrized by $t \in \mathbb{R}$.
Then we obtain a motion picture of the $2$-twist-spun trefoil depicted in Figure \ref{fig:motion_picture_of_two_twist_spun_trefoil}.
See also a movie at \url{http://www.youtube.com/watch?v=6lIM9p6XOKo} for more details.
(Strictly speaking, this is not a motion picture.
It consists of projection images of all pictures of a motion picture with over/under information.
However, we do not distinguish here them.)
Note that we can find the periodicity, whose period is $\pi$, in this motion picture.
\begin{figure}[htb]
 \begin{center}
  \includegraphics[scale=0.50]{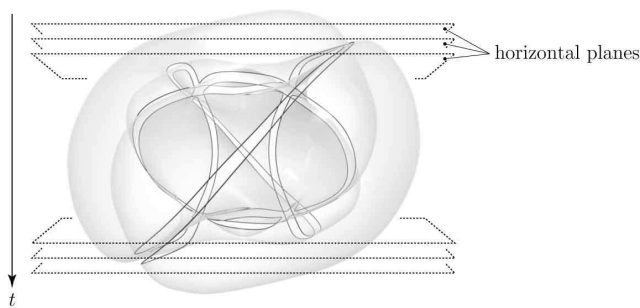}
 \end{center}
  \caption{Cut the $2$-twist-spun trefoil diagram along horizontal planes.}
 \label{fig:cut_diagram_along_horizontal_planes}
\end{figure}
\begin{figure}[htbp]
 \begin{center}
  \includegraphics[scale=0.67]{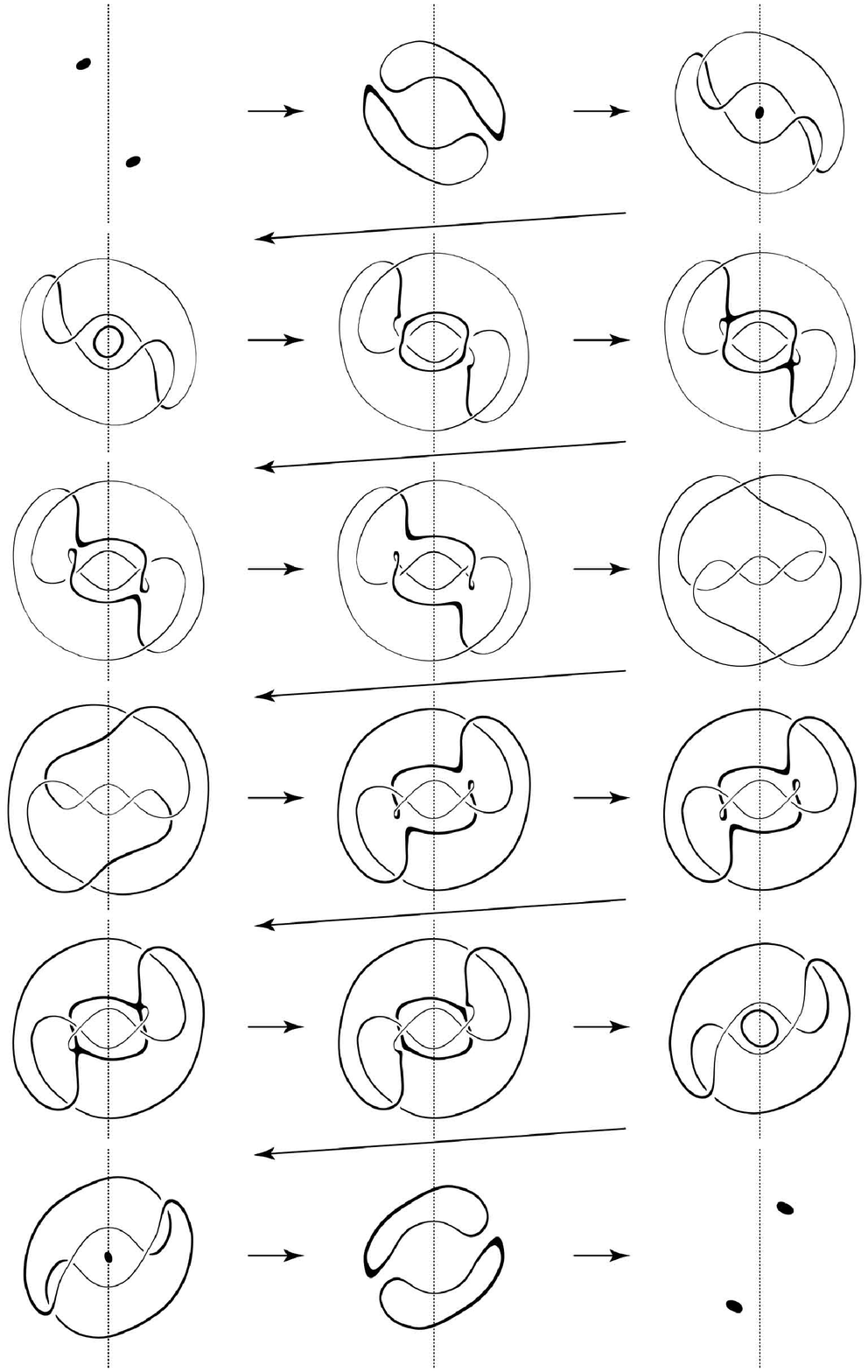}
 \end{center}
  \caption{A symmetric motion picture of the $2$-twist-spun trefoil}
 \label{fig:motion_picture_of_two_twist_spun_trefoil}
\end{figure}

The left half of this motion picture corresponds to the left half of the diagram depicted in Figure \ref{fig:planes_extending_radially}.
That is, a part of the diagram related with $\theta$ satisfying $0 \leq \theta \leq \pi$.
Recall that this part of the diagram presents one twist of a trefoil.
Therefore, if we lay out $n$ copies of this half of the motion picture (and this half of the diagram) in a circle compressing appropriately, then we obtain a motion picture (and a diagram) of the $n$-twist-spun trefoil.
Let us look at a part of the motion picture depicted in Figure \ref{fig:motion_picture_of_n_twist_spun_trefoil} that is sandwiched between two dotted lines.
Although it is compressed and normalized, this part coincides with the left half of the motion picture depicted in Figure \ref{fig:motion_picture_of_two_twist_spun_trefoil}.
Hence, the motion picture depicted in Figure \ref{fig:motion_picture_of_n_twist_spun_trefoil} presents the $n$-twist-spun trefoil.
\end{proof}

\subsection*{Acknowledgements}
The author would like to express his sincere gratitude to Professor Sadayoshi Kojima for encouraging him.
He is also grateful to Professor Shin Satoh for his invaluable comments.
He has used blender\footnote{\hskip 0.1em \url{http://www.blender.org/}}, which is a free open source and quite powerful 3D content creation suite, to construct, slice, and render $2$-knot diagrams.
This research has been supported in part by JSPS Global COE program ``Computationism as a Foundation for the Sciences''.

\bibliographystyle{amsplain}

\end{document}